\documentclass[12pt]{amsart}

 \usepackage{amsfonts,graphics,amsmath,amsthm,amsfonts,amscd,amssymb,amsmath,latexsym,multicol,
 mathrsfs}
\usepackage{epsfig,url}
\usepackage{flafter}
\usepackage{fancyhdr}
\usepackage{color}
\usepackage{hyperref}
\hypersetup{colorlinks=true, linkcolor=black}

\makeatletter

\def\jobis#1{FF\fi
  \def\predicate{#1}%
  \edef\predicate{\expandafter\strip@prefix\meaning\predicate}%
  \edef\job{\jobname}%
  \ifx\job\predicate
}

\makeatother

\if\jobis{proposal}%
\else
\fi

 \usepackage[matrix, arrow]{xy}

 \numberwithin{equation}{subsection}
 \numberwithin{footnote}{subsection}

 \newtheorem{lem}[subsection]{Lemma}
 \newtheorem{prop}[subsection]{Proposition}
 \newtheorem{thm}[subsection]{Theorem}
 \newtheorem{conj}[subsection]{Conjecture}

{
    \newtheoremstyle{upright}%
        {8pt plus2pt minus4pt}%
        {8pt plus2pt minus4pt}%
        {\upshape}%
        {}%
        {\bfseries\scshape}%
        {}%
        {1em}%
        {}%
\theoremstyle{upright}

 \newtheorem{defn}[subsection]{Definition}
 
 \newtheorem{exa}[subsection]{Example}

}

\title{On singularities of threefold weighted blowups}
\author{Yifei Chen}
\date{\today}
\begin{document}
\maketitle

\begin{abstract}
We answer a conjecture raised by Caucher Birkar of singularities of weighted blowups of $\mathbb{A}^n$ for $n\leq 3$.
\end{abstract}





\section{Introduction}

Throughout this paper, we work over the field of complex numbers $\mathbb{C}$.

The concept of complements is introduced by  Shokurov  when he studied log flips of threefolds (\cite{sh92}). The theory of complements is further developed in \cite{sh00},\cite{ps01},\cite{ps09}, \cite{bi19}. It is related to the boundedness of varieties and singularities of linear systems. It becomes a key ingredient of Birkar's BBAB Theorem, which confirms BAB conjecture (\cite{bi19},\cite{bi16}), and the construction of moduli of stable Fano varieties (\cite{blx19}, \cite{x19}).

\begin{thm}[Theorem 1.1 \cite{bi19}]
  Let $d$ be a natural number. Then there is a natural number $m$ depending only on $d$ such that if $X$ is any Fano variety of dimension $d$ with klt singularities, then there exists an $m$-complement for $K_X$.
\end{thm}

There is a general setting $(\epsilon,n)$-complements for log pairs $(X,B)$. we refer to (\cite{bi19} 2.18)  and (\cite{hls19}) for the definition and some related conjectures. It is conjectured the existence of such $(\epsilon,n)$-complements if $B\in \Gamma$, where $\Gamma\subseteq [0,1]$ is a DCC set. It is proved by Shokurov for the case $\epsilon=0$, $\dim X=2$ and $\Gamma$ is the standard set; by Prokhorov and Shokurov for the case $\epsilon=0,\dim X=3$ and $\Gamma$ is a hyperstandard set; by Birkar (\cite{bi19} Theorem 1.7, 1.8) for the case $\epsilon=0$, $\Gamma$ is a hyperstandard set; by  Han,  Liu and Shokurov (\cite{hls19} Teorem 1.14) for the case $\Gamma$ is a DCC set. Recently, Shokurov  (\cite{sh19}) announces a proof of the boundedness of complements in this general case, $\epsilon=0$ and $\Gamma=[0,1]$, under some conditions.  Shokurov (\cite{sh19}) also proves the existence of complements imply some important and well-known results, such as the ACC of log canonical thresholds, the finite generation of log canonical rings for klt pairs and the existence of flips. There are some recent progress for the theory of complements. We refer reads to \cite{fmx19}, \cite{hls19}, \cite{sh19} and \cite{xu19}  for more details.

The following conjecture due to Shokurov is one of major important open problems in birational geometry.

\begin{conj}[Shokurov]  \label{C:Shokurov}
  Let $d$ be a natural number, $\epsilon$ a positive real number. Then there is a natural number $n$ depending only on $d$ and $\epsilon$ such that if $X\rightarrow Z$ is a Fano contraction, that is $-K_X$ is ample $/Z$, and $X$ is $\epsilon$-lc, then there exists an $n$-klt complement for $K_X$ over any point $z\in Z$.
\end{conj}

The conjecture is known up to $\dim X=2$ (\cite{bi04}).

Birkar raises the following conjectures which are related to Conjecture \ref{C:Shokurov}.

Let $n\in \mathbb{N}_+$ be a positive integer, $a_1,a_2,\ldots,a_n$ coprime positive integers such that
$a_1\leq a_2\leq \cdots\leq a_n$. Let $X_{\textbf{a}}$ be the weighted blow up of the affine $n$-plane $\mathbb{A}^n$ at the origin with
the weight $\textbf{a}=(a_1, a_2,\ldots, a_n)$ (see Section \ref{S:wblowup}).

\begin{conj}[Birkar] \label{C:Birkar1} If $X_{\textbf{a}}$ has terminal singularities, is $a_1$ bounded?
\end{conj}

One related result with the Conjecture is the divisorial contraction of terminal threefolds  obtained by Kawakita.

\begin{thm}[\cite{k01} Theorem 1.1]
Let $Y$ be a $\mathbb{Q}$-factorial normal variety of dimension three with only terminal singularities,
and let $f:(Y\supset E)\rightarrow (X \ni P)$ be an algebraic germ of a divisorial contraction
which contracts its exceptional divisor $E$ to a smooth point $P$. Then $f$ is a weighted blowup.
More precisely, we can take local coordinates $x,y,z$ at $P$ and coprime positive integers $a$
and $b$, such that $f$ is the weighted blow-up of $X$ with its weights wt$(x,y,z)=(1,a,b)$.
\end{thm}

Very recently,  Sankaran (\cite{sa19}) has given an affirmative answer to Conjecture \ref{C:Birkar1} for the case $n=4$.

Compared with Conjecture \ref{C:Birkar1}, a more general conjecture of Birkar is:

\begin{conj}[Birkar] \label{Q:Birkar}
  For a fixed positive integer $n$ and a fixed positive real number
  $\epsilon\in (0,1]$, is there  a positive integer $M$ depending only on $n$ and $\epsilon$, such that if $X_{\textbf{a}}$
  is $\epsilon$-lc, then $a_1\leq M$, where $\textbf{a}=(a_1,a_2,\ldots,a_n)$, $a_1\leq a_2\leq \cdots \leq a_n$ and $a_1,a_2,\ldots,a_n$ are coprime.
\end{conj}

  One can not expect all $a_i$ are bounded, see Example \ref{E:unbounded_b}.

 Sankaran and  Santos recently claim that Conjecture \ref{Q:Birkar} holds for $\epsilon=1$ and any $n$. That is the case of canonical singularities.

In the paper, we study Conjecture \ref{Q:Birkar}, and get the following partial cases.

\begin{thm}[Main Theorem]
  Conjecture \ref{Q:Birkar} holds if:

  \begin{itemize}
    \item[1)]  $n=2$.
    \item[2)]  $n=3$.
    \item[3)]  $\frac{a_j}{a_2}\leq a_1^{\theta}$ for all $3\leq j\leq n$, where $\theta$ is a fixed real number such that $0<\theta<\frac{1}{2n^2}$.
  \end{itemize}
\end{thm}

The problem is a toric geometry problem. Conjecture \ref{Q:Birkar} can be reduced to a geometry of number problem, which is to prove the existence of an integral point in the interior of certain convex polytope. Then we can apply Dirichlet's approximation theorem (Theorem \ref{T:Diri_2}, Theorem \ref{T: Diri_S}). 

\subsection*{Acknowledgements}
Part of the work is done during the author visiting DPMMS at University of Cambridge. The author should thank Caucher Birkar's invitation, DPMMS at Cambridge for stimulating research environment.  The author thanks Dasheng Wei and Huayi Chen for many useful discussions. The author thanks Jingjun Han, Jihao Liu, Lu Qi and Chuyu Zhou for valuable comments. The author is supported by NSFC grants (No. 11688101, 11771426 and 11621061).

\section{Preliminaries}

\subsection{Pairs and singularities}
A \emph{log pair} $(X,B)$ consists of a normal variety $X$ and an $\mathbb{R}$-divisor $B\geq 0$ such that $K_X+B$ is $\mathbb{R}$-Cartier. An \emph{extraction}
is a proper birational morphism of normal varieties.

If $(X,B)$
is a log pair and $\mu:\tilde{X}\rightarrow X$ is an extraction, there exists a unique divisor $\tilde{B}$ on $\tilde{X}$ such that $\mu^*(K_X+B)=K_{\tilde{X}}
+\tilde{B}$ and $\tilde{B}=\mu^{-1}B$ on $\tilde{X}\setminus\text{Exc}(\mu)$. The identity $\tilde{B}=\sum_{E\subset \tilde{X}}(1-a(E;X,B))E$ associates
to each prime divisor $E$ of $\tilde{X}$ a real number $a(E;X,B)$, called the \emph{log discrepancy} of $E$ with respect to $(X,B)$. For simplicity, we write
mld$(E;B)$ for mld$(E;X,B)$.

\begin{defn}[\cite{sh88}] The \emph{minimal log discrepancy} of a log pair $(X,B)$ at a proper Grothendieck point $\eta\in X$ is defined as
$$\text{mld}(\eta;X,B)=\inf_{c_X(E)=\eta}a(E;X,B)$$
  where the infimum is taken after all prime divisors on extractions of $X$ having $\eta$ as a center on $X$. We set by definition mld$(\eta_X;X,B)=0$.
\end{defn}

The log pair $(X,B)$ has only \emph{lc (klt, $\epsilon$-lc) singularities} if mld$(\eta;B)\geq 0$ (mld$(\eta;B)>0$, mld$(\eta;B)\geq \epsilon$) for every
proper point $\eta\in X$. $(X,B)$ has only \emph{canonical} (\emph{terminal}) \emph{singularities} if mld$(\eta;B)\geq 1$ (mld$(\eta;B)>1$) for every point
$\eta\in X$ of codimension at least 2.

\subsection{Complements}

Let $X$ be a normal variety. We say that there exists an $n$-\emph{complement} ($n$-\emph{klt complement}) for $K_X$, if there exists a pair $(X,B)$, such that $(X,B)$ is lc (klt), and $n(K_X+B)\sim 0$.

Let $f:X\rightarrow Z$ be a contraction, that is, $f_*\mathcal{O}_X=\mathcal{O}_Z$. We say that there exists an $n$-\emph{klt complement for} $K_X$ \emph{over a point} $z\in Z$, if there exists a pair $(X,B)$, such that $(X,B)$ is klt, and $n(K_X+B)\sim 0/z\in Z$.

We refer readers to \cite{sh00}, \cite{ps01}, \cite{ps09}, \cite{bi04}, \cite{bi19}, \cite{fmx19}, \cite{hls19}, \cite{sh19} and \cite{xu19} for more background and applications of complements.

\subsection{Weighted blowups} \label{S:wblowup}

To be self-contained, we quote \cite{ksc04} 6.38 for the definition of weighted blow-ups.

Weighted blow-ups can be described in terms of charts in a manner similar to the standard blowup, but subject to a group action. Recall that if we fix one
of the standard affine charts of the blowup of the origin in $\mathbb{A}^n$, the blowing up morphism is given there by
$$\pi_i':\mathbb{A}^n\rightarrow \mathbb{A}^n$$
$$(x_1',\ldots,x_n')\mapsto (x_i'x_1',\ldots,x_i'x_{i-1}',x_i',x_i'x_{x+1}',\ldots,x_i'x_n').$$
In other words, $\pi_i'$ is given by the formulas
$$x_j=x_j'x_i',\text{\ \ \ if\ \ \ }j\neq i\text{\ \ \ and\ \ \ }x_i=x_i'.$$
To describer the analog for weighted blow-ups, fix relatively prime positive integers $a_1,\ldots,a_n$. For each $i$ between 1 and $n$, define a morphism
$p_i:\mathbb{A}^n\rightarrow \mathbb{A}^n$ by
$$x_j=x_j'(x_i')^{a_j}\text{\ \ \ if\ \ \ }j\neq i\text{\ \ \ and\ \  \ }x_i=(x_i')^{a_i}.$$
The map $p_i$ is birational if and only if $a_i=1$; more generally, it has degree $a_i$. Note that the map is well defined on the orbits of the
$\mathbb{Z}_{a_i}$-action
$$(x_i',\ldots,x_n')\mapsto (\zeta^{-a_1}x_1',\ldots,\zeta^{-a_{i-1}}x_{i-1}',\zeta x_i',\zeta^{-a_{i+1}}x_{i+1}',\ldots,\zeta^{-a_n}x_n'),$$
where $\mathbb{Z}_{a_i}$ is a cyclic group of order $a_i$ generated by a primitive $a_i$-th root of unity $\zeta$. Therefore, $p_i$ descends
to a birational morphism $\pi_i$ from the quotient variety:
$$\pi_i:\mathbb{A}^n/\mathbb{Z}_{a_i}\xrightarrow{\pi_i} \mathbb{A}^n.$$
These maps $\pi_i$ patch together to give a birational projective morphism $$\pi:X_{\textbf{a}}=B_{(a_1,\ldots,a_n)}\mathbb{A}^n\rightarrow \mathbb{A}^n.$$
This is the \emph{weighted blowup} of $\mathbb{A}^n$ with weights $\textbf{a}=(a_1,a_2,\ldots,a_n)$.

\subsection{Minimal log discrepancies of toric varieties} \label{S:mld_toric}
We refer the reader to \cite{fu93} for definitions and basic notations of toric geometry. Let $X=T_Nemb(\Delta)$ be a toroidal embedding, and let
$\{B_i\}_{i=1}^r$ be the $T_N$-invariant divisors of $X$, where $N$ is the lattice, corresponding to the primitive vectors $\{v_i\}_{i=1}^r$ on the 1-dimensional faces of $\Delta$.
Let $B=\sum_i(1-a_i)B_i$ be an invariant $\mathbb{R}$-divisor. Let $\varphi$ be the piecewise linear function such that $\varphi(v_i)=a_i$ for every $i$.

Under the above assumptions, we have the following formula for the minimal log discrepancies of $(X,B)$ at the generic points of the orbits (\cite{am99}):
$$\text{mld}(\eta_{\text{orb}(\sigma)};B)=\inf\{\varphi(v)|v\in \text{rel int}(\sigma)\cap N\},\sigma\in \Delta.$$
Here, rel int$(\sigma)$ denotes the relative interior of the cone $\sigma\subseteq \mathbb{R}\sigma$, and orb$(\sigma)$ is the $T_N$-orbit corresponding to the cone
$\sigma\in\Delta$.

\begin{thm}[\cite{am99} Theorem 4.1] In the above notations, let $X=\sqcup_{\sigma\in \Delta} \text{orb}(\sigma)$ be the partition of $X$ into
$T_N$-orbits. For every cone $\sigma\in \Delta$ and every closed point $x\in\text{orb}(\sigma)$, the minimal discrepancy
$$\text{mld}(x,B)=\text{mld}(\eta_{\text{orb}{(\sigma})};B)+\text{codim}(\sigma).$$
\end{thm}

\subsection{Weighted blowups and toric varieties}

By \cite{fu93}, it is easy to see that weighted blow-ups $X_{\textbf{a}}$ is a toric variety, where $\textbf{a}=(a_1,\ldots,a_n)$. The fan  $\Delta$ of $X_{\textbf{a}}$
is a subdivision of the first quadrant. Precisely, the ray through $(a_1,a_2,\ldots,a_n)$ subdivides the first quadrant into $n$ cones of dimension $n$, and the fan $\Delta$ is the collection of these $n$ cones of dimension $n$.

\begin{exa} The fan of $X_{(a,b)}$
\begin{center}
\includegraphics*[width=4cm]{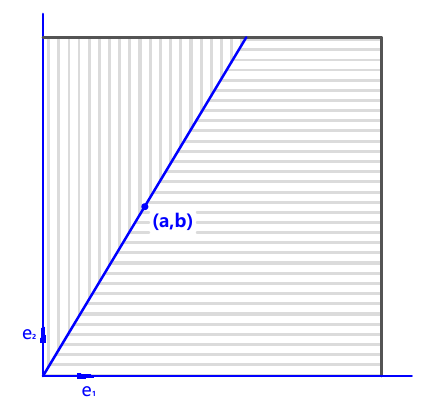}
\end{center}

\end{exa}

\begin{exa} The fan of $X_{(a,b,c)}$
\begin{center}
\includegraphics*[width=4cm]{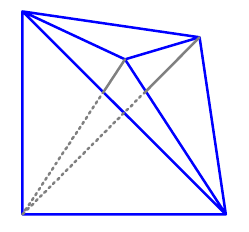}
\end{center}
\end{exa}

\subsection{Dirichlet's approximation theorem}

Dirichlet's Theorem is a fundamental result in Diophantine approxiation.

\begin{thm} \label{T:Diri_2} For any real number $\alpha$ and any positive integer $Z$, there exist integers $p$ and $q$ such that
$1\leq q\leq Z$ and $$|q\alpha-p|<\frac{1}{Z}.$$
\end{thm}

The simultaneous version of the Dirichlet's approximation theorem is:

\begin{thm}[Dirichlet's approximation theorem] \label{T: Diri_S}
 Given real numbers $\alpha_1,\ldots,\alpha_d$ and a positive integer $Z$, then there are integers $p_1,\ldots,p_d,q\in \mathbb{Z}, 1\leq q\leq Z$ such that
 $$\left|\alpha_i-\frac{p_i}{q}\right|\leq \frac{1}{qZ^{\frac{1}{d}}}.$$
\end{thm}

\subsection{Examples}

\begin{exa} \label{E:unbounded_b} $(X_{\textbf{n}},0)$, $\textbf{n}=(1,n)$, are $1$-lc for any positive integer $n$.

\begin{center}
\includegraphics*[width=4cm]{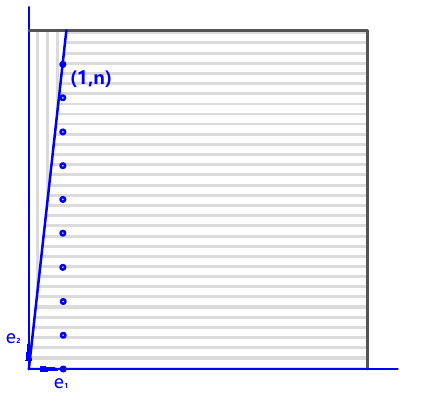}
\end{center}

Indeed, let $\sigma_1$ be the cone generated by $e_2=(0,1)$ and $v=(1,n)$. Then $\sigma_1$ is a smooth cone.  So mld$(x;0)=2$, where $x$
  is the toric invariant point corresponding to $\sigma_1$.

 Let $\sigma_2$ be the cone generated by $e_1=(1,0)$ and $v$.  The cone $\sigma_2$ is not smooth, but it has canonical singularities. We have mld$(y;0)=1$, where $y$ is the toric invariant point corresponding to $\sigma_2$.
  Indeed, the linear function $\varphi$ such that $\varphi(e_1)=\varphi(v)=1$ takes minimal values on the lattice points $w_i=(1,i),1\leq i\leq n-1$ in int $\sigma_2$
  and $\varphi(w_i)=1$.

  For any codimension 1 point $\eta$, mld$(\eta; 0)=1$.

  Therefore, $X_{\textbf{n}}$ are $1$-lc for any positive integer $n$.
\end{exa}

We can similar show the following example.

\begin{exa} $X_{\textbf{c}}$ is $1$-lc, where $\textbf{c}=(1,c_2,\ldots,c_n)$ for any $1\leq c_2\leq \cdots\leq c_n$ and gcd$(c_2,\ldots,c_n)=1$.
\end{exa}

\section{Reduction to a geometry of number problem}

\begin{lem} \label{L:GNP}
  Fix a positive integer $n$ and a positive real number $\epsilon\in (0,1]$. Let $C_n^{\epsilon}$ be the convex polytope with vertices $$\{\textbf{0},\epsilon \textbf{e}_1,
  \epsilon \textbf{e}_2,\ldots,\epsilon \textbf{e}_n,\epsilon \textbf{a}\},$$
  where $\{\textbf{e}_i\}_{i=1}^n$ is the standard basis of $\mathbb{R}^n$, $\textbf{a}=(a_1,\ldots,a_n)\in \mathbb{R}^n$, and $0\leq a_1\leq a_2\leq\cdots\leq a_n$.

  Then Conjecture \ref{Q:Birkar} holds, if there exists an integer $M$ depending only on $n$ and $\epsilon$, such that there exists an integral point in the interior of $C_n^{\epsilon}$ once $a_1\geq M$.
\end{lem}

\begin{exa}
The graphs of $C_2^{\epsilon}$ and $C_3^{\epsilon}$ are as follows

\begin{center}
\includegraphics*[width=4cm]{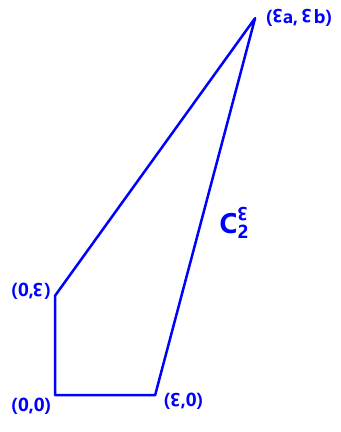}
\includegraphics*[width=3cm]{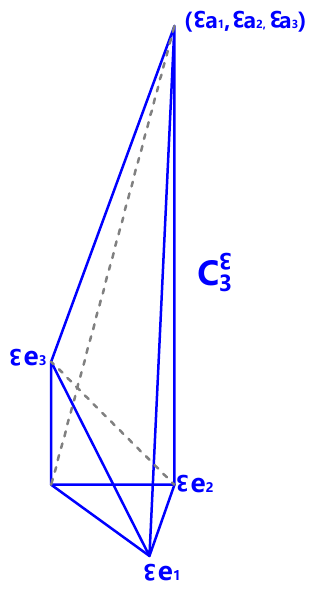}
\end{center}
\end{exa}

\begin{proof}[Proof of Lemma \ref{L:GNP}] Suppose there exists such an integer $M$ only depending on $n$ and $\epsilon$, that once $a_1\geq M$ then there would exist an integral point $\textbf{w}$ in the interior of $C_n^{\epsilon}$. By assumption, gcd$(a_1,\ldots,a_n)=1$, then the integral point $\textbf{w}$ lies in the interior of a unique cone generated by $\textbf{a}$ and $n-1$ vectors of $\{\textbf{e}_1,\ldots,\textbf{e}_n\}$. Without loss of generality, we can suppose that $\textbf{w}\in$ rel int $\tau$, where $\tau$ is the cone generated by $\{\textbf{a},\textbf{e}_1,\ldots,\textbf{e}_{n-1}\}$. Moreover, since $\textbf{w}$ is in the interior of $\tau\cap C_n^{\epsilon}$, we have
$$\textbf{w}=\sum_{i=1}^{n-1}\lambda_i\cdot \epsilon \textbf{e}_i+\lambda_0\cdot \epsilon \textbf{a},$$
  where $0\leq \lambda_i\leq 1$ for $i=0,\ldots,n-1$ and $0<\sum_{i=0}^{n-1}\lambda_i<1$.

  View $X_{\textbf{a}}$ as the log pair $(X_{\textbf{a}},0)$. Let $\varphi$ be the piecewise linear function with $\varphi(\textbf{e}_i)=1$ for $i=1,\ldots,n-1$ and $\varphi(\textbf{a})=1$. Therefore,
  $$\varphi(\textbf{w})=\sum_{i=1}^{n-1}\epsilon \lambda_i\varphi(\textbf{e}_i)+\epsilon \lambda_0 \varphi(\textbf{a})=\epsilon \sum_{i=0}^{n-1}\lambda_i<\epsilon.$$

  Apply the formula of minimal log discrepancies of toric varieties in Section \ref{S:mld_toric}, we get
  $$\text{mld}(\eta_{\text{orb}(\tau)})=\inf\{\varphi(\textbf{v})|\textbf{v}\in \text{rel int}(\tau)\cap N\}\leq \varphi(\textbf{w})<\epsilon.$$
  Hence $X_{\textbf{a}}$ is not $\epsilon$-lc, if $a_1\geq M$.
\end{proof}

\section{Proof of Main Theorem for $n=2$}

As we discussed in Lemma \ref{L:GNP}, to prove Main Theorem for $n=2$, it suffices to prove Lemma \ref{L:GNP} for $n=2$. We shall see that $M$ can be taken $\lfloor (\frac{2}{\epsilon}+1)^2\rfloor +1$, where $\lfloor\ \rfloor$ is the round down function.

\begin{proof}[Proof of Lemma \ref{L:GNP} for $n=2$]

Let $$L_1: y-\epsilon=\frac{b-1}{a}x$$ be the line passing through the points $(0,\epsilon)$ and $(\epsilon a,\epsilon b)$,
$$L_2: y=\frac{b}{a-1}(x-\epsilon)$$
the line passing through the points $(\epsilon,0)$ and $(\epsilon a,\epsilon b)$.
Let $\alpha=\frac{b}{a}$, and $Z=\lfloor\sqrt{a}\rfloor$. Then $\alpha\geq 1$ by assumption that $a\leq b$. Apply Dirichlet approximation Theorem \ref{T:Diri_2}, there exist integers
$p$ and $q$ such that $1\leq q\leq Z$ and  $|\alpha-\frac{p}{q}|<\frac{1}{qZ}$.

Let $$L:y=\frac{p}{q}x$$ be a line, and the point $(x_0,y_0)$ be the intersection point of $L$ and $L_1$ or $L_2$. The idea is to show that if $a$
is large enough, then the $x_0>q$. So there exists at least one integral point on the line $L$, and thus an integral point in the interior of $C_2^{\epsilon}$.

\medskip

1) \underline{\emph{Case 1.}} $\frac{p}{q}\leq \frac{b}{a}$. In the case, $L$ intersects with $L_2$.

\begin{center}
\includegraphics*[width=4cm]{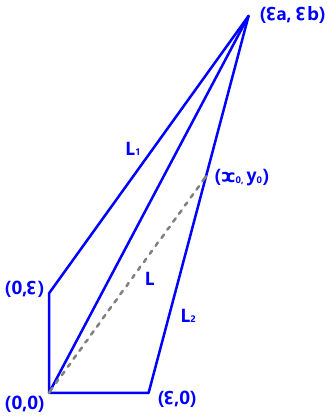}
\end{center}

Notice that $\frac{p}{q}>\frac{b}{a}-\frac{1}{qZ}$ and $qZ\leq Z^2\leq a$. So
$$\frac{b}{a-1}>\frac{b}{a}\geq \frac{p}{q}>\frac{b}{a}-\frac{1}{qZ}>0.$$

In the case, $$\begin{array}{rcl}
  x_0&=&\frac{\epsilon \frac{b}{a-1}}{\frac{b}{a-1}-\frac{p}{q}}>\frac{\epsilon\frac{b}{a-1}}{\frac{b}{a-1}-\left(\frac{b}{a}-\frac{1}{qZ}\right)}\\
  \\
  &=&
\frac{\epsilon}{1-\frac{a-1}{b}\cdot \frac{b}{a}+\frac{a-1}{b}\cdot \frac{1}{qZ}}=\frac{\epsilon}{\frac{1}{a}+\frac{a-1}{b}\frac{1}{qZ}}\\
\\
&>&\frac{\epsilon}{\frac{1}{a}+\frac{1}{qZ}}\geq q\frac{Z\epsilon}{2}
\end{array}
$$
The last inequality follows from $qZ\leq Z^2\leq a$.

\smallskip

Let $M=\lfloor(\frac{2}{\epsilon}+1)^2\rfloor+1$. If $a>M$, then $\sqrt{a}>\frac{2}{\epsilon}+1$ and thus $$x_0>q\frac{Z\epsilon}{2}=q\frac{\lfloor\sqrt{a} \rfloor
\epsilon}{2}>q.$$
Hence there exists at least a lattice point in the interior of $C_2^{\epsilon}$.

\medskip

2) \underline{\emph{Case 2.}} $\frac{p}{q}> \frac{b}{a}$. In the case, $L$ intersects with $L_1$.

\begin{center}
\includegraphics*[width=4cm]{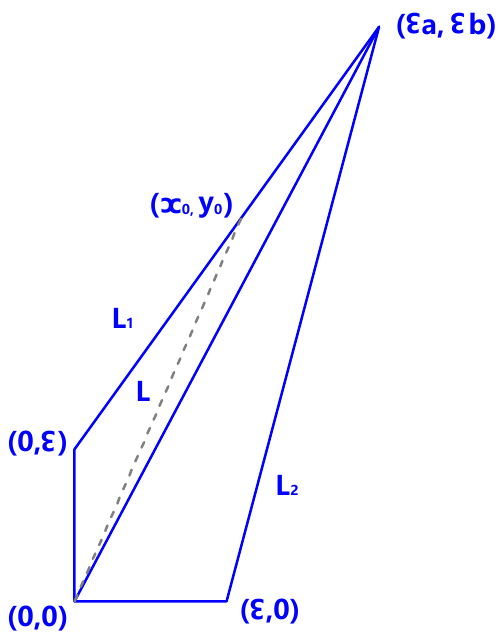}
\end{center}

Notice that $\frac{p}{q}<\frac{b}{a}+\frac{1}{qZ}$ and $qZ\leq Z^2\leq a$.

In the case, $$x_0=\frac{\epsilon}{\frac{p}{q}-\frac{b-1}{a}}>\frac{\epsilon}{\frac{b}{a}+\frac{1}{qZ}-\frac{b-1}{a}}=\frac{\epsilon}{\frac{1}{a}+\frac{1}{qZ}}
\geq q\frac{Z\epsilon}{2}$$

If $a>M=\lfloor(\frac{2}{\epsilon}+1)^2\rfloor+1$, then $\sqrt{a}>\frac{2}{\epsilon}+1$ and thus $$x_0>q\frac{Z\epsilon}{2}=q\frac{\lfloor\sqrt{a} \rfloor
\epsilon}{2}>q.$$
Hence there exists at least a lattice point in the interior of $C_2^{\epsilon}$.
\end{proof}

\section{Proof of Main Theorem 3)}

In the section, we consider general case of $n$.

Let $\Pi_i,1\leq i\leq n,$ be the hyperplane passing through the set of vectors
$\{\epsilon \textbf{e}_1,\epsilon \textbf{e}_2,\ldots,\epsilon\textbf{e}_n,\epsilon \textbf{a}\}\setminus\{\epsilon \textbf{e}_i\}$. Then the equation for $\Pi_i$ is:
$$\frac{\sum_{1\leq j\neq i\leq n}a_j-1}{a_i}x_i=\sum_{1\leq j\neq i\leq n}x_j-\epsilon.$$

Apply Dirichlet Approximation, Theorem \ref{T: Diri_S}, to the given real numbers $\frac{a_2}{a_1},\ldots,\frac{a_n}{a_1}$ and the natural number $Z=\lfloor
a_1^{\frac{1}{n}}\rfloor$, then there are integers $p_2,\ldots,p_{n}$ and $1\leq p_1\leq Z$, such that $$\left|\frac{a_i}{a_1}-\frac{p_i}{p_1}\right|\leq
 \frac{1}{p_1 Z^{\frac{1}{n-1}}},\text{\ \ \ for \ }2\leq i\leq n,$$
 and $p_1Z^{\frac{1}{n-1}}\leq Z^\frac{n}{n-1}\leq a_1^{\frac{1}{n-1}}\leq a_1$.

Let $L$ be the line defined by $$\left\{\begin{array}{rcl}
  x_2&=&\frac{p_2}{p_1}x_1\\
  &\vdots&\\
  x_n&=&\frac{p_n}{p_1}x_1
\end{array}\right.$$

Let $\lambda_1=0,\lambda_i=\frac{p_j}{p_1}-\frac{a_j}{a_1},j=2,\ldots,n$. By Dirichlet Approximation, $|\lambda_j|<\frac{1}{p_1Z^{\frac{1}{n-1}}}$.

Rewrite the equations of $\Pi_i$, we get
$$\sum_{j\neq i}(x_j-\frac{a_j}{a_i}x_i)+\frac{x_i}{a_i}=\sum_{j\neq i}(x_j-\frac{a_j}{a_1}x_1-
\frac{a_j}{a_i}(x_i-\frac{a_i}{a_1}x_1))+\frac{x_i}{a_i}=\epsilon.$$

The line $L$ intersects the hyperplane $\Pi_i$ for some $i$. Let $(x_1^0,x_2^0,\ldots,x_n^0)$ be the intersection point. Plug the equation of $L$ into
the equation of $\Pi_i$, we get
$$\sum_{j\neq i}\left((\frac{p_j}{p_1}-\frac{a_j}{a_1})x_1-\frac{a_j}{a_i}(\frac{p_i}{p_1}-\frac{a_i}{a_1})x_1\right)+\frac{p_i}{p_1}\frac{x_1}{a_i}=\epsilon.$$
Let $$A_i=\sum_{j \neq i}(\lambda_j-\frac{a_j}{a_i}\lambda_i)+\frac{1}{a_i}(\frac{a_i}{a_1}+\lambda_i)=\sum_{j \neq i}(\lambda_j-\frac{a_j}{a_i}\lambda_i)
+\frac{1}{a_1}+\frac{\lambda_i}{a_i}.$$
Then $x_1^0=\frac{\epsilon}{A_i}$.

The goal is to show there exists an integer $M$ depending on $n$ and $\epsilon$, such that if $a_1>M$, then $x_1^0>p_1$. Therefore, there exists at least one integral
point in the interior of the convex body $C_n^{\epsilon}$.

\begin{prop}[=Main Theorem 3)] \label{P:MainTh_3)}
 Suppose $\frac{a_j}{a_2}\leq a_1^{\theta}$ for all $3\leq j\leq n$, where $\theta$ is a real number  such that
$0<\theta<\frac{1}{2n^2}$. Then Conjecture \ref{Q:Birkar} holds.
\end{prop}

\begin{proof}
Since $a_2\leq a_3\leq \cdots\leq a_n$, we get $\frac{a_j}{a_i}\leq a_1^{\theta}$ for $i\neq 1$.

If $i=1$, then $$A_1=\sum_{j\neq 1}\lambda_j+\frac{1}{a_1}<\frac{n-1}{p_1Z^{\frac{1}{n-1}}}+\frac{1}{a_1}\leq \frac{n}{p_1Z^{\frac{1}{n-1}}}$$
since $p_1Z^{\frac{1}{n-1}}<a_1$.

If $i\neq 1$, then $$0<A_i=\sum_{j \neq i}(\lambda_j-\frac{a_j}{a_i}\lambda_i)
+\frac{1}{a_1}+\frac{\lambda_i}{a_i}\leq \frac{n-1}{p_1Z^{\frac{1}{n-1}}}+(n-1)a_1^{\theta}\frac{1}{p_1Z^{\frac{1}{n-1}}}+\frac{1}{a_1}
+\frac{1}{a_1}\frac{1}{p_1Z^{\frac{1}{n-1}}}$$
since $\lambda_j<\frac{1}{p_1Z^{\frac{1}{n-1}}}$ for all $j$, $\frac{a_j}{a_i}\leq a_1^{\theta}$ and $a_i\geq a_1$ for all $i\neq 1$,

Since $a_1\geq 1$ and $p_1Z^{\frac{1}{n-1}}\leq a_1$, we have $$\frac{n-1}{p_1Z^{\frac{1}{n-1}}}\leq (n-1)a_1^{\theta}\frac{1}{p_1Z^{\frac{1}{n-1}}},$$
$$\frac{1}{a_1}\leq a_1^{\theta}\frac{1}{p_1Z^{\frac{1}{n-1}}},$$
$$\frac{1}{a_1p_1Z^{\frac{1}{n-1}}}\leq a_1^{\theta}\frac{1}{p_1Z^{\frac{1}{n-1}}}.$$

Therefore, $A_i\leq 2na_1^{\theta}\frac{1}{p_1Z^{\frac{1}{n-1}}}$ for all $i$.

So $$\frac{x_1^0}{p_1}=\frac{\epsilon}{p_1A_i}\geq \frac{\epsilon Z^{\frac{1}{n-1}}}{2na_1^{\theta}}>\frac{\epsilon(a_1^{\frac{1}{n}}-1)^{\frac{1}{n-1}}}{2n
a_1^{\theta}}>\frac{\epsilon}{4n}a_1^{\frac{1}{n(n-1)}-\theta}$$
As $a_1$ is sufficiently large, $(a_1^{\frac{1}{n}}-1)^{\frac{1}{n-1}}>\frac{1}{2}a_1^{\frac{1}{n(n-1)}}$. Since $\theta<\frac{1}{2n^2}$, $\frac{1}{n(n-1)}
-\theta>0$. Therefore, as $a_1$ sufficiently large, $\frac{x_1^0}{p_1}$ turns to infinity.
\end{proof}

\section{Proof of Main Theorem for $n=3$}

By Proposition \ref{P:MainTh_3)}, to prove Main Theorem for $n=3$, we only need to prove for the case $\frac{a_3}{a_2}>a_1^{\theta}$. In the case, visually the height of $C_3^{\epsilon}$, which is $a_3$, grows much faster than $a_2$. Then we reduce the case to the 2-dim case. Precisely, projecting $C_3^{\epsilon}$ onto the $x_1x_2$-plane. By induction, there exists at least one lattice point
$(q,p)$ in the projection as $a_1$ is sufficiently large. Then show that the interval of $C_3^{\epsilon}$ and the height line $L': x_1=q,x_2=p$ is greater than
1 as $a_1$ is sufficiently large. Hence there is a lattice point inside the convex body $C_3^{\epsilon}$ as $a_1$ is sufficiently large.

\begin{center}
\includegraphics*[width=8cm]{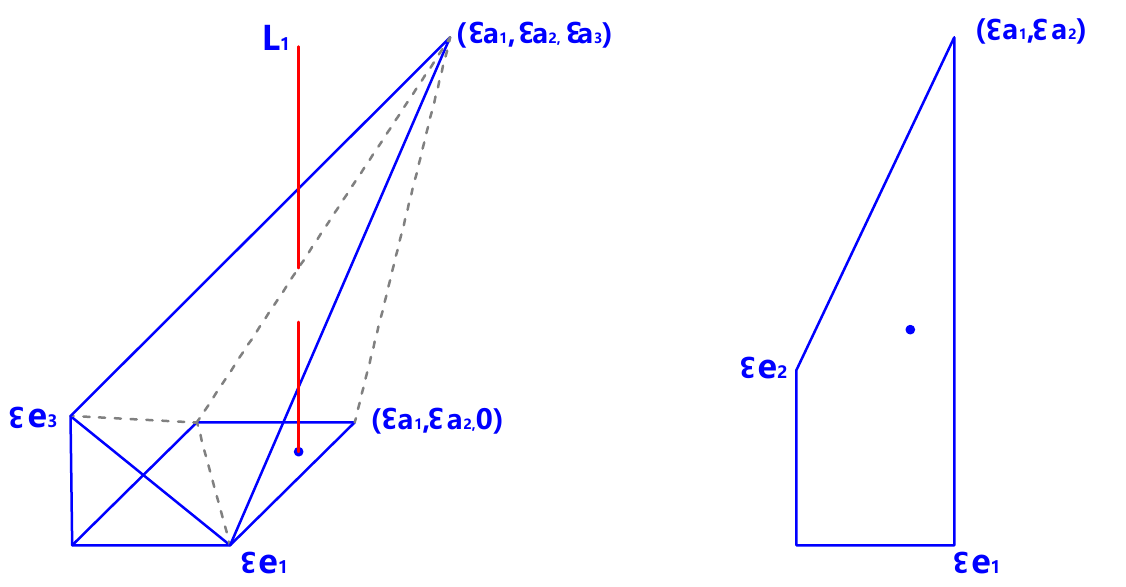}
\end{center}

\begin{proof}
Project $C_3^{\epsilon}$ onto the $x_1x_2$-plane. The convex set is as same as in 2-dimensional case. By Dirichlet approximation, let $M_2=\lfloor
\sqrt{a_1}\rfloor$, then there exist integers
$p$ and $q$ such that $$\left|\frac{a_2}{a_1}-\frac{p}{q}\right|<\frac{1}{qM_2},$$
and $1\leq q\leq M_2$. So $qM_2<a_1$.

For $a_1$ sufficiently large, the lattice point $(q,p)$ is in the projection. Let the line $L':x_1=q,x_2=p$. Then $L'$ intersects
$$\Pi_3:\frac{a_1+a_2-1}{a_3}x_3=x_1+x_2-\epsilon,$$
and $L'$ will intersect one of $$\Pi_1:\frac{a_2+a_3-1}{a_1}x_1=x_2+x_3-\epsilon$$ or $$\Pi_2:\frac{a_1+a_3-1}{a_2}x_2=x_1+x_3-\epsilon$$
 The goal is to show the interval $|L'_{23}|$ of $L'$ between $\Pi_2$ and $\Pi_3$, or the interval
$|L'_{13}|$ of $L'$ between $\Pi_1$ and $\Pi_3$ has length greater than 1 as $a_1$ is sufficiently large. Then there exists at least one lattice point
inside $C_3^{\epsilon}$.

A direct computation shows that
$$|L'_{13}|=\frac{a_2+a_3-1}{a_1}q-p+\epsilon-\frac{p+q-\epsilon}{a_1+a_2-1}a_3=(\frac{a_2-1}{a_1}q-p+\epsilon)+(\frac{q}{a_1}-\frac{p+q-
\epsilon}{a_1+a_2-1})a_3$$
Let $\lambda_2=-\frac{a_2}{a_1}+\frac{p}{q}$. Then $|\lambda_2|<\frac{1}{qM_2}$ and $qM_2<a_1$ by Dirichlet approximation.
So $$|L'_{13}|=-\frac{q}{a_1}-q\lambda_2+\epsilon+(\epsilon-\frac{q}{a_1}-\lambda_2q)\frac{a_3}{a_1+a_2-1}
\geq -\frac{1}{M_2}-\frac{1}{M_2}+\epsilon+(\epsilon-\frac{1}{M_2}-\frac{1}{M_2})\frac{a_3}{2a_2}$$
Since $\frac{a_3}{a_2}\geq a_1^{\theta}$ by assumption, we get
$$|L'_{13}| \geq(\epsilon-\frac{2}{M_2})+(\epsilon-\frac{2}{M_2})
\frac{a_1^{\theta}}{2}.$$
As $a_1$ turns to infinity, $|L'_{13}|$ will turn to infinity.

Similarly, a direct computation shows that
$$|L'_{23}|=\frac{a_1+a_3-1}{a_2}p-q+\epsilon-\frac{p+q-\epsilon}{a_1+a_2-1}a_3=(\frac{a_1-1}{a_2}p-q+\epsilon)+(\frac{p}{a_2}-
\frac{p+q-\epsilon}{a_1+a_2-1})a_3$$
Let $\lambda_2=-\frac{a_2}{a_1}+\frac{p}{q}$. Then $|\lambda_2|<\frac{1}{qM_2}$ and $qM_2<a_1$ by Dirichlet approximation.
So $$|L'_{23}|=-\frac{q}{a_1}+\frac{a_1-1}{a_2}q\lambda_2+\epsilon+(\epsilon-\frac{p}{a_2}+\frac{a_1}{a_2}q\lambda_2)\frac{a_3}{a_1+a_2-1}$$
Since $\lambda_2=\frac{p}{q}-\frac{a_2}{a_1}\leq \frac{1}{qM_2}$, we get $\frac{p}{a_2}\leq \frac{q}{a_1}+\frac{1}{a_2M_2}$.

Hence $$|L'_{23}|\geq (-\frac{1}{M_2}-\frac{1}{M_2}+\epsilon)+(\epsilon-\frac{q}{a_1}-\frac{1}{a_2M_2}-\frac{1}{M_2})\frac{a_3}{2a_2}
\geq (\epsilon-\frac{2}{M_2})+(\epsilon-\frac{2}{M_2}-\frac{1}{a_2M_2})\frac{a_1^{\theta}}{2}$$
As $a_1$ turns to infinity, $|L'_{23}|$ turns to infinity.
\end{proof}

\vspace{1cm}

Hua Loo-Keng Key Laboratory of Mathematics, Academy of Mathematics and Systems Science,
Chinese Academy of Sciences, No. 55 Zhonguancun East Road, Haidian District, Beijing, 100190, P.R.China\\
yifeichen@amss.ac.cn\\\\


\begin{thebibliography}{LONGEST}
\bibitem[Am99]{am99} F. Ambro; \emph{On minimal log discrepancies}, Mathematical Research Letters 6, 573 -- 580 (1999).

\bibitem[Bi04]{bi04} C. Birkar; \emph{Topics in Modern Algebraic Geometry}, PhD thesis, https://www.dpmms.cam.ac.uk/~cb496/finalthesis.pdf

\bibitem[Bi19]{bi19} C. Birkar; \emph{Anti-pluricanonical systems on Fano varieties}, Annals of Mathematics 190 (2019) 345 -- 463.

\bibitem[Bi16]{bi16} C. Birkar; \emph{Singularities of linear systems and boundedness of Fano varieties}, arXiv:1609.05543v1

\bibitem[BLX19]{blx19} H. Blum, Y. Liu and C. Xu; \emph{Openness of K-semistability for Fano varieties}, arXiv:1907.02408.

\bibitem[FMX19]{fmx19} S. Filipazzi, J. Moraga and Y. Xu; \emph{Log canonical 3-fold complements},  	arXiv:1909.10098

\bibitem[Fu93]{fu93} W. Fulton; \emph{Introduction to toric varieties}, Annals of Mathematics Studies, 131. The William H. Rover Lectures in Geometry.
Princeton University Press, Princeton, NJ, 1993.

\bibitem[HLS19]{hls19} J. Han, J. Liu and V. V. Shokurov; \emph{ACC for minimal log discrepancies of exceptional singularities}, arXiv:1903.04338v1

\bibitem[KSC04]{ksc04} J. Koll\'{a}r, K. E. Smith and A. Corti, \emph{Rational and Nearly Rational Varieties}

\bibitem[K01]{k01} M. Kawakita; \emph{Divisorial contractions in dimension three which contract divisors to smooth points}, Invent. Math. 145,
105--119 (2001).

\bibitem[Pr01]{pr01} Yu. Prokhorov; \emph{Lecture on compelements on log surfaces}. MSJ Memoirs, 10. Mathematical Society of Japan, Tokyo, 2001. viii +130 pp. ISBN: 4-931469-12-4.

\bibitem[PS01]{ps01} Yu. Prokhorov and V. V. Shokurov; \emph{The first fundamental theorem on complements: from global to local.} (Russian) Izv. Ross. Akad. Nauk Ser. Mat. 65 (2001), no. 6, 99 -- 128; translation in Izv. Math. 65 (2001), no. 6, 1169 -- 1196.

\bibitem[PS09]{ps09}    Yu. Prokhorov and V. V. Shokurov; \emph{Towards the second main theorem on complements}. J. Algebraic Geom. 18 (2009), no. 1, 151 -- 199.

\bibitem[Sa19]{sa19} G. K. Sankaran; \emph{Singularities of fourfold blowups}, arXiv:1911.06435v1

\bibitem[Sh88]{sh88} V. V. Shokurov, \emph{Problems about Fano varieties}, Birational Geometry of Algebraic Varieties, Open problems --- Katata 1988, 30 -- 32.

\bibitem[Sh92]{sh92} V. V. Shokurov, \emph{3-fold log flips}, Izv. Russ. A.N. Ser. Mat. 56: 105--203, 1992.

\bibitem[Sh93]{sh93} V. V. Shokurov, \emph{Three-dimensional log flips}. with an appendix in English by Yujiro Kawamata. Russian Acad. Sci. Izv. Math. 40 (1993), no. 1, 95  -- 202.

\bibitem[Sh00]{sh00} V. V. Shokurov, \emph{Complements on surfaces}. Algebraic geometry, 10, J. Math. Sci. (New York) 102 (2000), no. 2, 3876 -- 3932.

\bibitem[Sh19]{sh19} V. V. Shokurov, \emph{Boundedness and existence of $n$-complements}, 2019, (preprint to appear).

\bibitem[X19]{x19} C. Xu, \emph{A minimizing valuation is quasi-monomial}, to appear in Annals of Math

\bibitem[Xu19]{xu19} Y. Xu; \emph{Complements on log canonical Fano varieties},  	arXiv:1901.03891
\end{thebibliography}
\end{document}